\documentclass[12pt]{article}

\usepackage{bbm}
\usepackage{amssymb}
\usepackage{amsmath}
\usepackage{amsthm}
\usepackage{amsmath}
\usepackage{amsthm}
\usepackage{amsfonts}

\DeclareMathOperator{\1}{\mathbbm{1}}

\newcommand{\eee}{{\rm e}}

\newcommand{\mmp}{\mathbb{P}}

\newcommand{\dod}{\overset{{\rm d}}{\to}}

\newcommand{\me}{\mathbb{E}}
\newcommand{\mr}{\mathbb{R}}
\newcommand{\mn}{\mathbb{N}}

\newtheorem{thm}{Theorem}[section]
\newtheorem{lemma}[thm]{Lemma}

\newtheorem{assertion}[thm]{Proposition}

\theoremstyle{definition}

\theoremstyle{remark}
\newtheorem{rem}[thm]{Remark}
\begin{document}
\title{Some central limit theorems for critical beta-splitting trees\footnote{This article is partly based on the unpublished manuscript `Another proof of CLT for critical beta-splitting tree' by A. Iksanov written in May, 2024. That manuscript contained a new proof of the joint distributional convergence of $(L_n, D_n)$. New results on the distributional convergence of $((L_{n,1},L_{n,2},\ldots), L_n, D_n)$ were obtained by A. Nikitin and R. Yakymiv.}}
\date{\today}

\author{Alexander Iksanov\footnote{Faculty of Computer Science and Cybernetics, Taras Shevchenko National University of Kyiv, Ukraine, e-mail:
iksan@univ.kiev.ua} \ \ Anatolii Nikitin\footnote{Faculty of Natural Sciences, Jan Kochanowski University of Kielce, Poland and Faculty of Economics, National University of Ostroh Academy, Ukraine; e-mail address: anatolii.nikitin@ujk.edu.pl} \ \  Roman Yakymiv\footnote{Faculty of Computer Science and Cybernetics, Taras Shevchenko National University of Kyiv, Ukraine, e-mail:
yakymiv@knu.ua}}

\maketitle

\begin{abstract}
We further explore a connection initially unveiled in Iksanov (2025) between critical beta-splitting trees and infinite `balls-in-boxes' schemes. Using the connection, we derive a new joint central limit theorem for components of the height of a leaf chosen uniformly at random in the discrete version of a critical beta-splitting tree. Also, we obtain a joint central limit theorem for the heights in the discrete and continuous versions of a critical beta-splitting tree.
\end{abstract}



\section{Setting and result}

For $\beta\in (-2,+\infty)$ and integer $m\geq 2$, define the probability distribution $q^{(\beta)}(m):=(q^{(\beta)}(m,i),\,1\leq i\leq m-1)$ by $$q^{(\beta)}(m,i):=\frac{1}{a^{(\beta)}(m)}\frac{\Gamma(\beta+i+1)\Gamma(\beta+m-i+1)}{\Gamma(i+1)\Gamma(m-i+1)},\quad i=1,2,\ldots, m-1,$$ where $\Gamma$ is the Euler gamma function and $a^{(\beta)}(m)$ is the normalizing constant. Following \cite{Aldous:1996}, we construct a random rooted tree called {\it beta-splitting tree with $n$ leaves} by recursively splitting the set $\{1,2,\ldots, n\}$ of `elements'. There is a left edge at the root leading to a left subtree with $G_n$ elements $\{1,2,\ldots, G_n\}$ and a right edge leading to a right subtree with $n-G_n$ elements $\{G_n+1, G_n+2,\ldots, n\}$. Here, $G_n$ is a random variable with distribution $q^{(\beta)}(n)$. Recursively, each set of $m$ elements with $m\geq 2$ is split into two subsets, with a left subset having a random size with distribution $q^{(\beta)}(m)$. The procedure stops once all sets of elements have become leaves (singletons). Here is a possible realization of the procedure with $n=5$:
\begin{equation}\label{eq:example}
[12345] \to [12][345]\to [1][2][34][5]\to [1][2][3][4][5].
\end{equation}

The beta-splitting tree with $\beta=-1$ is called {\it critical}. For the critical trees, $$q^{(-1)}(m,i)=\frac{m}{2h_{m-1}}\frac{1}{i(m-i)},\quad i=1,2,\ldots, m-1,$$ where $h_m:=\sum_{j=1}^m 1/j$. The term `critical' stems from the fact the leaf heights are of order $n^{-\beta-1}$ for $\beta\neq -1$, whereas these are of order $\log n$ for $\beta=-1$. A detailed investigation of the critical beta-splitting trees was undertaken very recently in a series of papers \cite{Aldous+Janson:2025, Aldous+Janson:2025b, Aldous+Pittel:2025}, 
see also \cite{Aldous+Janson+Li:2025, Brandenberger+Chin+Mossel:2025, Iksanov:2025, Kolesnik:2025}. The state of the art can be traced via the dynamically changing arXiv article \cite{Aldous+Janson:2025+}.

There are two versions of this construction: discrete and continuous. In the discrete construction, the edges are assumed to have lengths $1$, whereas in the continuous construction, the edge emanating from a subset of size $m\geq 2$ has a random length with the exponential distribution of mean $1/h_{m-1}$. Let $L_n$ and $D_n$ denote the heights of a leaf chosen uniformly at random in the discrete and continuous versions of a critical beta-splitting tree with $n$ leaves, respectively. The terms `edge-height' for $L_n$ and `time-height' for $D_n$ are also used. As the splitting procedure proceeds, the sizes of subtrees containing a uniformly chosen leaf decrease from $n$ to $1$. For $j\in\mn$, $j\leq n-1$, denote by $L_{n,j}$ the number of the size $j$ decrements in a sequence of the sizes of those subtrees. Referring back to \eqref{eq:example} and assuming that a uniformly chosen leaf is $3$, the sizes of subtrees are $5$, $3$, $2$ and $1$, $L_5=3$, $L_{5,1}=2$, $L_{5,2}=1$ and $L_{5,3}=L_{5,4}=0$. Plainly, $L_n=\sum_{j=1}^{n-1}L_{n,j}$ and $\sum_{j=1}^{n-1}jL_{n,j}=n-1$.

In Theorem 1.7 of \cite{Aldous+Pittel:2025}, the following central limit theorems were proved. Let ${\rm Normal}\,(0,1)$ denote a random variable with the normal distribution of zero mean and unit variance. As usual, $\dod$ denotes convergence in distribution. 
\begin{assertion}\label{thm:main}
As $n\to\infty$, $$\frac{L_n-(2\zeta(2))^{-1}(\log n)^2}{((2\zeta(3)/3(\zeta(2))^3)(\log n)^3 )^{1/2}}~\dod~{\rm Normal}\,(0,1),\quad n\to\infty,$$ where $\zeta(s):=\sum_{n\geq 1}n^{-s}$ for $s>1$, and  $$\frac{D_n-(\zeta(2))^{-1}\log n}{((2\zeta(3)/(\zeta(2))^3)\log n)^{1/2}}~\dod~{\rm Normal}\,(0,1),\quad n\to\infty.$$
\end{assertion}
The proof in \cite{Aldous+Pittel:2025} was based on an asymptotic analysis of recurrences satisfied by $u\mapsto \me \eee^{uL_n}$ and $u\mapsto \me \eee^{uD_n}$ for $u\geq 0$ and $n\in\mn$, respectively. The recurrences are actually consequences of the distributional equalities
\begin{equation}\label{eq:recur}
\quad L_n\overset{{\rm d}}{=} L_{n-J_n}+1,\quad n\geq 2,\quad D_n\overset{{\rm d}}{=} D_{n-J_n}+\tau_n,\quad n\geq 2,\quad D_1=L_1=0,
\end{equation}
where $\overset{{\rm d}}{=}$ denotes equality of distributions. Here, $J_n$ is independent of $(L_k, D_k)_{k\geq 2}$ on the right-hand sides and has distribution $$\mmp\{J_n=k\}=\frac{1}{kh_{n-1}},\quad k=1,2,\ldots, n-1,$$ and $\tau_n$ is independent of $J_n$ and $(D_k)_{k\geq 2}$ and has the exponential distribution of mean $1/h_{n-1}$. In a more recent article \cite{Kolesnik:2025}, the central limit theorem for $L_n$ was derived from Theorem 2.1 in \cite{Neininger+Rueschendorf:2004}, which provides sufficient conditions ensuring the validity of a central limit theorem for distributional solutions to recurrences like \eqref{eq:recur} and more general ones. Even more recently, the scope of that approach was extended in \cite{Brandenberger+Chin+Mossel:2025}. Among other things, this enabled the authors of \cite{Brandenberger+Chin+Mossel:2025} to prove a central limit theorem for the total length of a continuous-time critical beta-splitting tree with $n$ leaves as $n\to\infty$.

The purpose of this note is to prove the joint convergence of $(L_{n,1},\ldots, L_{n,j}, L_n, D_n)$, properly centered and normalized, as $n\to\infty$ for each $j\in\mn$.
\begin{thm}\label{thm:main2}
As $n\to\infty$,
\begin{equation}\label{eq:lndn}
\Big(\frac{L_n-(2\zeta(2))^{-1}(\log n)^2}{(2\zeta(3)(\zeta(2))^{-3}(\log n)^3 )^{1/2}}, \frac{D_n-(\zeta(2))^{-1}\log n}{(2\zeta(3)(\zeta(2))^{-3}\log n)^{1/2}}\Big)~\dod~\Big(\int_0^1 W^\ast(u){\rm d}u, W^\ast(1)\Big),
\end{equation}
where $W^\ast:=(W^\ast(u))_{u\in [0,1]}$ is a standard Brownian motion.

For any fixed $j\in\mn$, as $n\to\infty$,
\begin{equation}\label{eq:lnr}
\Big(\frac{L_{n,r}-(\zeta(2)r)^{-1}\log n}{(\log n)^{1/2}}\Big)_{1\leq r\leq j}~\dod~\big((2\zeta(3)(\zeta(2))^{-3})^{1/2}r^{-1} W^\ast(1)+(r\zeta(2))^{-1/2}Y_r\big)_{1\leq r\leq j}
\end{equation}
for the same Brownian motion $W^\ast$ as in \eqref{eq:lndn}. Here, $Y_1$, $Y_2,\ldots, Y_j$ are independent random variables with the standard normal distribution, which are independent of $W^\ast$. Also, limit relations \eqref{eq:lndn} and \eqref{eq:lnr} hold jointly.
\end{thm}
\begin{rem}
Observe that $\me\big[\big(\int_0^1 W^\ast(u){\rm d}u\big)^2\big]=1/3$. This explains the fact that the denominator of the fraction involving $L_n$ in Theorem \ref{thm:main2} differs by the factor $3^{-1/2}$ from that in Proposition \ref{thm:main}.
\end{rem}

It was shown in \cite{Iksanov:2025} that, for each $n\in\mn$, there is a connection 
between beta-splitting tree with $n$ leaves and certain `balls-in-boxes' scheme with $n$ balls and infinitely many boxes. Using that connection, 
discussed in some detail in Section \ref{sect:coupling}, we provide in Section \ref{sect:proofs} a reasonably simple proof of Theorem \ref{thm:main2}.

\section{Our strategy}\label{sect:coupling}

It is important for what follows that the distributional equalities in \eqref{eq:recur} hold true jointly, also together with the corresponding distributional recurrences for $(L_{n,1},\ldots, L_{n,j})$, $j\in\mn$, $j\leq n-1$. Fix such a $j$. For $n\geq 1$, put $X_n:=(L_{n,1},\ldots, L_{n,j}, L_n, D_n)$ and $Y_n:=(\1_{\{J_n=1\}},\ldots, \1_{\{J_n=j\}}, 1, \tau_n)$. Then
\begin{equation}\label{eq:recur1}
X_n\overset{{\rm d}}{=} X_{n-J_n}+Y_n,\quad n\geq 2, \quad X_1=\underbrace{(0,0, 0,\ldots, 0)}_{j+2~\text{zeros}}.
\end{equation}

Let $S:=(S(t))_{t\geq 0}$ be a subordinator (an almost surely increasing L{\'e}vy process) with $S(0)=0$, zero drift, no killing and the L{\'e}vy measure $\nu$ defined by
\begin{equation}\label{eq:measure}
\nu({\rm d}x)=\frac{\eee^{-x}}{1-\eee^{-x}}\1_{(0,\infty)}(x)\,{\rm d}x.
\end{equation}
Let $E_1,\ldots, E_n$ be independent random variables with the exponential distribution of unit mean, which are independent of $S$. The closed range of $S$ 
has zero Lebesgue measure and splits the positive halfline into infinitely many disjoint intervals that we call gaps. To obtain the aforementioned `balls-in-boxes' scheme, take the gaps in the role of boxes and the points of the exponential sample in the role of balls. Denote by $K_n$ the number of gaps occupied by at least one $E_j$, $j=1,2,\ldots, n$. Also, for $j\in\mn$, $j\leq n$, denote by $K_{n,j}$ the number of gaps occupied by $j$ 
elements of the exponential sample $E_1,\ldots, E_n$. Note that $K_n=\sum_{j=1}^n K_{n,j}$ and $\sum_{j=1}^n jK_{n,j}=n$.  

Put $S^\leftarrow(t):=\inf\{u\geq 0: S(u)>t\}$ for $t\geq 0$ and $E_{n,n}:=\max (E_1,\ldots, E_n)$ for $n\in\mn$. Here is the basic observation which enables us to connect critical beta-splitting trees with the `balls-in-boxes' scheme. 
\begin{assertion}\label{prop:eqdistr}
For each $n\geq 1$ and each $j\in\mn$, $j\leq n$, $$(L_{n+1, 1},\ldots, L_{n+1,j}, L_{n+1}, D_{n+1})\overset{{\rm d}}{=} (K_{n,1},\ldots, K_{n,j}, K_n, S^\leftarrow (E_{n,n})).$$
\end{assertion}

With this at hand, limit relation \eqref{eq:lndn} in Theorem \ref{thm:main2} is a consequence of the following result, in which $\nu$ is a nonzero L\'{e}vy measure, not necessarily the same as given in \eqref{eq:measure}. Put
$$\Phi(t):=\int_{(0,\,\infty)}\big(1-\exp(-t(1-\eee^{-x}))\big)\nu({\rm  d}x),\quad t>0,$$ $${\tt m_2}:={\rm Var}\,[S(1)]=\int_{(0,\,\infty)}x^2\nu({\rm d}x), ~~~{\tt m_1}:={\mathbb E}[S(1)]=
\int_{(0,\,\infty)} x \nu({\rm d}x).$$
\begin{thm}\label{thm:31a}
Assume that $t\mapsto \Phi(\eee^t)$ is regularly varying at $\infty$ of positive index $\beta$ and that ${\tt m_2}\in (0, \infty)$.
Then, as $n\to\infty$,
\begin{multline*}
\Big(\frac{K_n-{\tt m}_1^{-1}\int_1^ n y^{-1}\Phi(y){\rm d}y}{({\tt m_2}{\tt m}^{-3}_1\log n)^{1/2}\Phi(n)}, \frac{S^\leftarrow (E_{n,n})-{\tt m}_1^{-1}\log n}{({\tt m}_2{\tt m}_1^{-3}\log n)^{1/2}}\Big)\\~\dod~\Big(
\beta \int_0^1 W(1-u)u^{\beta-1}{\rm d}u, W(1)\Big),
\end{multline*}
where $W:=\big(W(u)\big)_{u\in[0,1]}$ is a standard Brownian motion.
\end{thm}

Now let $\nu$ be as in \eqref{eq:measure}. According to a Hurwitz identity (see, for instance, formula (23.2.7) in \cite{Abram}), $${\tt m_r}:=\int_{(0,\infty)}x^r \nu({\rm d}x)=\int_0^1 x^{-1}(-\log (1-x))^r{\rm d}x=\Gamma(r+1)\zeta(r+1),\quad r>0,$$ where $\Gamma$ is the Euler gamma function. In particular, ${\tt m_1}:=\zeta(2)$ and ${\tt m_2}=2\zeta(3)$. Further,
\begin{multline}\label{eq:Phi}
\Phi(t)=\int_{(0,\,\infty)}\big(1-\exp\{-t(1-\eee^{-x})\}\big)\nu({\rm  d}x)=\int_0^\infty 
\big(1-\exp\{-t(1-\eee^{-x})\}\big)\frac{\eee^{-x}}{1-\eee^{-x}}{\rm  d}x\\=\int_0^t y^{-1}(1-\eee^{-y}){\rm d}y=\log t+\gamma+O(\eee^{-t}),\quad t\to\infty,
\end{multline}
where $\gamma=\int_0^1 y^{-1}(1-\eee^{-y}){\rm d}y-\int_1^\infty y^{-1}\eee^{-y}{\rm d}y$ is the Euler-Mascheroni constant. This shows that $\beta$ in Theorem \ref{thm:31a} is equal to $1$ and $$\int_1^ n y^{-1}\Phi(y){\rm d}y=(\log n)^2/2+\gamma \log n+o(\log n)=(\log n)^2/2+o((\log n)^{3/2}),\quad n\to\infty.$$ Thus, limit relation \eqref{eq:lndn} in Theorem \ref{thm:main2} does indeed follow from Proposition \ref{prop:eqdistr} and Theorem \ref{thm:31a}.

In view of Proposition \ref{prop:eqdistr}, limit relation \eqref{eq:lnr} and the joint convergence in Theorem \ref{thm:main2} follow from Proposition \ref{prop:gnedin} given next. We write $\Phi^{(r)}$ to denote the $r$th derivative of $\Phi$.
\begin{assertion}\label{prop:gnedin}
Let $j\in\mn$. Assume that
\begin{equation}\label{eq:gneassum}
\Phi(t)=\log t+c+O(t^{-\varepsilon}),\quad t\to\infty
\end{equation}
for some constants $c\in\mr$ and $\varepsilon>0$, and that $\nu$ has a continuous density on $(0,\infty)$ and satisfies $\nu([y,\infty])=O(\eee^{-\delta y})$ as $y\to\infty$ for some $\delta>0$. Assume also that, for each $r=1,\ldots, j$, the function $t\mapsto (-1)^{r+1}\Phi^{(r)}(t)t^r$ is nondecreasing on $(1,\infty)$. Then, as $n\to\infty$,
\begin{equation}\label{eq:knr}
\Big(\frac{K_{n,r}-({\tt m}_1 r)^{-1}\log n}{(\log n)^{1/2}}\Big)_{1\leq r\leq j}~\dod~\big(({\tt m}_2{\tt m}_1^{-3})^{1/2}r^{-1} W(1)+(r{\tt m}_1)^{-1/2}Z_r\big)_{1\leq r\leq j}
\end{equation}
for the same Brownian motion $W$ as in Theorem \ref{thm:31a}. Here, $Z_1$, $Z_2,\ldots, Z_j$ are independent random variables with the standard normal distribution, which are independent of $W$. Also, limit relations of Theorem \ref{thm:31a} and Proposition \ref{prop:gnedin} hold jointly, namely, as $n\to\infty$,
\begin{multline}\label{eq:joint}
\Big(\Big(\frac{K_{n,r}-({\tt m}_1 r)^{-1}\log n}{(\log n)^{1/2}}\Big)_{1\leq r\leq j}, \frac{K_n-(2{\tt m}_1)^{-1}(\log n)^2}{(\log n)^{3/2}}, \frac{S^\leftarrow (E_{n,n})-{\tt m}_1^{-1}\log n}{(\log n)^{1/2}}\Big)\\~\dod~\Big(\big(({\tt m}_2{\tt m}_1^{-3})^{1/2}r^{-1} W(1)+(r{\tt m}_1)^{-1/2}Z_r\big)_{1\leq r\leq j}, ({\tt m}_2{\tt m}_1^{-3})^{1/2}\int_0^1 W(u){\rm d}u, ({\tt m}_2{\tt m}_1^{-3})^{1/2} W(1)\Big).
\end{multline}
\end{assertion}

We shall derive relation \eqref{eq:knr} from Theorem 15 in \cite{Gnedin etal:2006}. The distributional convergence of $K_n$ as stated in Proposition \ref{prop:gnedin} was earlier obtained in Theorem 12 of \cite{Gnedin etal:2006}. The proof of this result given in \cite{Gnedin etal:2006} is different from our proof of Theorem \ref{thm:31a}.

The measure $\nu$ as in \eqref{eq:measure} has a continuous density on $(0,\infty)$ and satisfies $\nu([y,\infty])\sim \eee^{-y}$ as $y\to\infty$. According to \eqref{eq:Phi}, the corresponding $\Phi$ satisfies \eqref{eq:gneassum} with $c=\gamma$ and any $\varepsilon>0$. Further, for $r\in\mn$, put $F_r(t):=1-\eee^{-t}(1+t+t^2/2+\ldots+t^{r-1}/(r-1)!)$ for $t\geq 0$. Thus, $F_r$ is the distribution function of a gamma distributed random variable with parameters $r$ and $1$. It can be checked with the help of mathematical induction that $$(-1)^{r+1}t^r\Phi^{(r)}(t)=(r-1)!F_r(t),\quad t\geq 0.$$ In particular, the function $t\mapsto  (-1)^{r+1}t^r\Phi^{(r)}(t)$ is increasing on $[0,\infty)$. Thus, Proposition \ref{prop:gnedin} applies to the `balls-in-boxes' scheme with $\nu$ as in \eqref{eq:measure}.

\section{Proofs}\label{sect:proofs}

\begin{proof}[Proof of Proposition \ref{prop:eqdistr}]
For fixed $n\in\mn$ and each $t\geq 0$, put $$N_n(t):=\#\{1\leq k\leq n: E_k>S(t)\}.$$ Then $N_n:=(N_n(t))_{t\geq 0}$ is a decreasing integer-valued continuous-time Markov chain which starts at $n$ and gets absorbed at $0$. The variable $S^\leftarrow(E_{n,n})$ is the absorption time at $0$. According to Theorem 5.2(i) in \cite{Gnedin+Pitman:2005}, $N_n$ goes from $m$ to $m-k$ with probability
\begin{multline*}
\frac{{m \choose k}\int_{(0,\infty)}(1-\eee^{-x})^k \eee^{-x(m-k)}\nu({\rm d}x)}{\sum_{k=1}^m {m \choose k}\int_{(0,\infty)}(1-\eee^{-x})^k \eee^{-x(m-k)}\nu({\rm d}x)}=\frac{{m \choose k}\int_0^1 x^{k-1}(1-x)^{m-k}{\rm d}x}{\sum_{k=1}^m {m \choose k}\int_0^1 x^{k-1}(1-x)^{m-k}{\rm d}x}\\=\frac{1}{kh_m},\quad k=1,2,\ldots, m.
\end{multline*}
Thus, denoting by $J_n^\ast$ the number of exponential points (out of $n$) falling into the left-most occupied gap we infer $$\mmp\{J_n^\ast=k\}=\frac{1}{kh_n}=\mmp\{J_{n+1}=k\},\quad k=1,2,\ldots, n.$$ Let $\tau_n^\ast$ be the time of the first decrement of $N_n$. This random variable has the exponential distribution with mean $$\frac{1}{\sum_{k=1}^n {n\choose k}\int_{(0,\infty)}(1-\eee^{-x})^k \eee^{-x(n-k)}\nu({\rm d}x)}=\frac{1}{h_n}.$$ Fix any $j\in\mn$, $j\leq n$. Now, with $X_n^\ast:=(K_{n,1},\ldots, K_{n,j}, K_n, S^\leftarrow(E_{n,n}))$ and $Y_n^\ast:=(\1_{\{J_n^\ast=1\}},\ldots, \1_{\{J_n^\ast=j\}}, 1, \tau_n^\ast)$ for $n\geq 1$, we infer $$X_n^\ast \overset{{\rm d}}{=} X^\ast_{n-J^\ast_n}+Y^\ast_n,\quad n\geq 1, \quad X^\ast_0=\underbrace{(0,0, 0,\ldots, 0)}_{j+2~\text{zeros}},$$ where $J_n^\ast$ is independent of $Y_n^\ast$ and $(X^\ast_k)_{k\geq 1}$ on the right-hand side. Comparing this distributional equality with \eqref{eq:recur1} completes the proof of the proposition.
\end{proof}

\begin{proof}[Proof of Theorem \ref{thm:31a}]
Let $D([0,1])$ be the Skorohod space of c\`{a}dl\`{a}g functions defined on $[0,1]$ equipped with the $J_1$-topology. In what follows, $\Rightarrow$ denotes weak convergence on this space. It is a standard fact that $$\Big(\frac{S(ut)-{\tt m}_1ut}{({\tt m}_2t)^{1/2}}\Big)_{u\in [0,1]}~\Rightarrow~(-W(u))_{u\in [0,1]},\quad t\to\infty$$ (it is more convenient to write here $-W(u)$ rather than $W(u)$). By inversion (see, for instance, Lemma 1 in \cite{Vervaat:1972})
\begin{equation}\label{eq:conv}
\Big(\frac{S^\leftarrow(ut)-{\tt m}^{-1}_1ut}{({\tt m}_2{\tt m}^{-3}_1 t)^{1/2}}\Big)_{u\in [0,1]}~\Rightarrow~(W(u))_{u\in [0,1]},\quad t\to\infty.
\end{equation}
Specializing this to one-dimensional convergence and recalling that $E_{n,n}-\log n$ converges in distribution to a random variable with the Gumbel distribution we conclude that $$\frac{S^\leftarrow(E_{n,n})-{\tt m}^{-1}_1\log n}{({\tt m}_2{\tt m}^{-3}_1 \log n)^{1/2}}~\dod~W(1),\quad n\to\infty.$$ This demonstrates the distributional convergence of the second coordinate in Theorem \ref{thm:31a}. The distributional convergence of the first coordinate was proved in Theorem 3.1(a) of \cite{Gnedin+Iksanov:2012}. Thus, we only need to combine these into the joint convergence.

For each $t>0$, put $$W_t(u):=\frac{S^\leftarrow (t(1-u))-{\tt m}_1^{-1}t(1-u)}{({\tt m}_2{\tt m}^{-3}_1 t)^{1/2}},\quad u\in [0,1].$$ An inspection of the proof of Theorem 3.1(a) of \cite{Gnedin+Iksanov:2012} reveals that the distributional convergence of the first coordinate in Theorem \ref{thm:31a} is driven by that of $\int_{[0,\,1]}W_t(u){\rm d}_u(\Phi(\eee^{tu})/\Phi(\eee^t))$. Thus, we are left with showing that, for all real $\alpha_1$ and $\alpha_2$,
\begin{multline*}
\alpha_1 W_t(0)+\alpha_2\int_{[0,\,1]}W_t(u){\rm d}_u(\Phi(\eee^{tu})/\Phi(\eee^t))=\int_{[0,\,1]}(\alpha_1W_t(0)+\alpha_2 W_t(u)){\rm d}_u(\Phi(\eee^{tu})/\Phi(\eee^t))\\ ~\dod~\beta\int_0^1(\alpha_1W(1)+\alpha_2 W(1-u))u^{\beta-1}{\rm d}u=\alpha_1 W(1)+\alpha_2\beta\int_0^1 W(1-u)u^{\beta-1}{\rm d}u,\quad t\to\infty.
\end{multline*}
Since, for each $t>0$, $u\mapsto \Phi(\eee^{tu})/\Phi(\eee^t)$, $u\in[0,1]$ is the distribution function of a nonnegative random variable, $\lim_{t\to\infty}(\Phi(\eee^{tu})/\Phi(\eee^t))=u^\beta$ for each $u\geq 0$ and, in view of \eqref{eq:conv}, $(\alpha_1 W_t(0)+\alpha_2 W_t(u))_{u\in [0,1]}\Rightarrow (\alpha_1 W(1)+\alpha_2 W(1-u))_{u\in [0,1]}$ as $t\to\infty$, this convergence is secured by Lemma \ref{lem:conv} (a). The proof of Theorem \ref{thm:31a} is complete.
\end{proof}

\begin{proof}[Proof of Proposition \ref{prop:gnedin}]
Let $\pi:= (\pi(t))_{t\geq 0}$ denote a Poisson process on $[0,+\infty)$ of unit intensity. For $r\in\mn$ and $t\geq 0$, denote by $\mathcal{K}(t,r)$ the number of gaps that contain $r$ elements of the Poissonized sample $E_1,\ldots, E_{\pi(t)}$.

By Theorem 15 in \cite{Gnedin etal:2006}, under the assumptions of Proposition \ref{prop:gnedin}, excluding the monotonicity assumption, for any fixed $j\in\mn$, $$\Big(\frac{K_{n,r}-({\tt m}_1 r)^{-1}\log n}{(\log n)^{1/2}}\Big)_{1\leq r\leq j}~\dod~ (X_r)_{1\leq r\leq j},\quad n\to\infty,$$ where $(X_r)_{1\leq r\leq j}$ is a centered Gaussian vector with $\me [X_rX_s]={\tt m}_2{\tt m}_1^{-3}(rs)^{-1}+\1_{\{r=s\}}({\tt m}_1r)^{-1}$ for $r,s=1,2,\ldots, j$. Also,
\begin{equation}\label{eq:relation}
\Big(\frac{\mathcal{K}(t,r)-({\tt m}_1 r)^{-1}\log t}{(\log t)^{1/2}}\Big)_{1\leq r\leq j}~\dod~ (X_r)_{1\leq r\leq j},\quad t\to\infty,
\end{equation}
see p.~596 in \cite{Gnedin etal:2006}. The formula for the covariance entails that $$X_r=({\tt m}_2{\tt m}_1^{-3})^{1/2}r^{-1}Z+({\tt m}_1r)^{-1/2}Z_r,\quad r=1,2,\ldots, j,$$ where $Z$, $Z_1,\ldots, Z_j$ are independent random variables with the standard normal distribution. It remains to identify $Z$ with $W(1)$, where $W$ is a Brownian motion as in Theorem \ref{thm:31a}. To this end, we use a decomposition, for $r=1,\ldots, j$, \begin{multline*}
\mathcal{K}(\eee^t,r)-{\tt m}_1^{-1}\int_0^t h_r(y){\rm d}y=\Big(\mathcal{K}(\eee^t,r)-\int_{[0,\,t]}h_r(t-x){\rm d}S^\leftarrow(x)\Big)\\+\Big(\int_{[0,\,t]}h_r(t-x){\rm d}S^\leftarrow(x)-{\tt m}_1^{-1}\int_0^t h_r(y){\rm d}y\Big) =:A(t,r)+B(t,r),
\end{multline*}
where $h_r(t):=(-1)^{r+1}(r!)^{-1}\Phi^{(r)}(\eee^t)\eee^{rt}$ for $t\in\mr$. We shall show that
\begin{equation}\label{eq:purp}
\Big(\frac{B(t,r)}{t^{1/2}}\Big)_{1\leq r\leq j}~\dod~ ({\tt m}_2{\tt m}^{-3}_1)^{1/2}W(1)(r^{-1})_{1\leq r\leq j},\quad t\to\infty,
\end{equation}
where $(W(u))_{u\in [0,1]}$ is a standard Brownian motion as in \eqref{eq:conv}, and that 
\begin{equation}\label{eq:expect}
\int_0^t h_r(y){\rm d}y=r^{-1}\log t+O(1),\quad t\to\infty.
\end{equation}
In view of \eqref{eq:relation}, we then have $$\Big(\frac{A(t,r)}{t^{1/2}}\Big)_{1\leq r\leq j}~\dod~(({\tt m}_1r)^{-1/2}Z_r)_{1\leq  r\leq j},\quad t\to\infty.$$ Finally, these together with the argument given in the proof of Theorem \ref{thm:31a} will justify formula \eqref{eq:joint}.

\noindent {\sc Proof of \eqref{eq:purp}}. We intend to check that, for any $\alpha_1$, $\alpha_2,\ldots, \alpha_j\in\mr$, $$\frac{\sum_{r=1}^j\alpha_r\int_{[0,\,t]}h_r(t-x){\rm d}(S^\leftarrow(x)-{\tt m}_1^{-1}x)}{({\tt m}_2{\tt m}^{-3}_1 t)^{1/2}}~\dod~W(1)\sum_{r=1}^j\alpha_r r^{-1},\quad t\to\infty.$$ 
Integrating by parts we conclude that the left-hand side is equal to $$\sum_{r=1}^j\alpha_r\Big(h_r(0)\frac{S^\leftarrow(t)-{\tt m}_1^{-1}t}{({\tt m}_2{\tt m}^{-3}_1 t)^{1/2}}+\int_{(0,\,1]}\frac{S^\leftarrow(t(1-x))-{\tt m}_1^{-1}t(1-x)}{({\tt m}_2{\tt m}^{-3}_1 t)^{1/2}}{\rm d}_xh_r(tx)\Big).$$ For $r=1,\ldots, j$, the functions $h_r$ are nondecreasing by assumption and satisfy $\lim_{t\to\infty}h_r(t)=r^{-1}$ by Lemma \ref{lem:tech}. Also, relation \eqref{eq:conv} entails
\begin{equation*}
\Big(\frac{S^\leftarrow(t(1-x))-{\tt m}^{-1}_1t(1-x)}{({\tt m}_2{\tt m}^{-3}_1 t)^{1/2}}\Big)_{x\in [0,1]}~\Rightarrow~(W(1-x))_{x\in [0,1]},\quad t\to\infty
\end{equation*}
in the $J_1$-topology on $D([0,1])$. With these at hand, an application of Lemma \ref{lem:conv}(b) ensures that
\begin{multline*}
\sum_{r=1}^j\alpha_r\Big(h_r(0)\frac{S^\leftarrow(t)-{\tt m}_1^{-1}t}{({\tt m}_2{\tt m}^{-3}_1 t)^{1/2}}+\int_{(0,\,1]}\frac{S^\leftarrow(t(1-x))-{\tt m}_1^{-1}t(1-x)}{({\tt m}_2{\tt m}^{-3}_1 t)^{1/2}}{\rm d}_xh_r(tx)\Big)\\~\dod~W(1)\sum_{r=1}^j\alpha_r h_r(0)+ W(1)\sum_{r=1}^j\alpha_r(r^{-1}-h_r(0))=W(1)\sum_{r=1}^j\alpha_r r^{-1}.
\end{multline*}

\noindent {\sc Proof of \eqref{eq:expect}}. We use mathematical induction. If $r=1$, then $$\int_0^t h_1(y){\rm d}y=\Phi(\eee^t)-\Phi(1)=\log t+O(1),\quad t\to\infty.$$ Assume that \eqref{eq:expect} holds with $r=s$. Then using this together with $h_s(t)=O(1)$ as $t\to\infty$, we infer $$\int_0^t h_{s+1}(y){\rm d}y=h_s(t)-h_s(0)+s(s+1)^{-1}\int_0^t h_s(y){\rm d}y=(s+1)^{-1}\log t+O(1),\quad t\to\infty.$$

The proof of Proposition \ref{prop:gnedin} is complete. 
\end{proof}

\section{Appendix}

Here, we collect a couple of auxiliary results. Recall that $$\Phi(t)=\int_{(0,\,\infty)}(1-\exp(-t(1-\eee^{-x})))\nu({\rm d}x),\quad t>0.$$
\begin{lemma}\label{lem:tech}
Let $c>0$. The relation
\begin{equation}\label{eq:haan}
\lim_{t\to\infty}(\Phi(\lambda t)-\Phi(t))=c\log \lambda
\end{equation}
for all $\lambda>0$ implies that, 
for each $r\in\mn$,
\begin{equation}\label{eq:deriv}
\lim_{t\to\infty}(-1)^{r+1}t^r \Phi^{(r)}(t)=c(r-1)!.
\end{equation}
Conversely, if relation \eqref{eq:deriv} holds for some $r\in\mn$, then \eqref{eq:haan} holds.
\end{lemma}
\begin{proof}
Since $\Phi^\prime$ is a nondecreasing function, relation \eqref{eq:haan} is equivalent to $\lim_{t\to\infty}t \Phi^\prime(t)=c$ by Theorem 3.6.8 in \cite{BGT:1989}. For the converse part of the lemma, assume that \eqref{eq:deriv} holds for some $r\geq 2$. Applying Karamata's theorem (Proposition 1.5.10 in \cite{BGT:1989}) $r-1$ times we obtain \eqref{eq:deriv} with $r=1$.

For the direct part, assume that \eqref{eq:deriv} holds with $r=1$. Now we shall prove by mathematical induction that it holds for all $r\in\mn$. Indeed, let \eqref{eq:deriv} hold with $r=k$. By Bernstein's criterion, the function $\Phi^\prime$ is completely monotone. In particular, $(-1)^{k+2}\Phi^{(k+1)}$ is a nonincreasing function. Now \eqref{eq:deriv} with $r=k+1$ follows from a version of the monotone density theorem (see the remark following Theorem 1.7.2 in \cite{BGT:1989}).
\end{proof}

The second result is essentially Lemma 6.4.2 in \cite{Iksanov:2016}. 
\begin{lemma}\label{lem:conv}
Let $a,b>0$. Assume that $X_t\Rightarrow X$ as $t\to\infty$ on $D([0,a])$ in the $J_1$-topology.

\noindent (a) If $(\nu_t)_{t\geq 0}$ is a family of finite measures such that $\nu_t$ converge weakly to $\nu$ as $t\to\infty$, where $\nu$ is a finite measure on $[0,a]$ which is continuous with respect to Lebesgue measure, then $$\int_{[0,\,a]}X_t(y)\nu_t({\rm d}y)~\dod~\int_{[0,\,a]}X(y)\nu({\rm d}y),\quad t\to\infty.$$ 

\noindent (b) If $f:[0,\infty)\to [0,\infty)$ is a nondecreasing function satisfying $\lim_{t\to\infty}f(t)=b$, then $$\int_{(0,\,a]}X_t(y){\rm d}_yf(ty)~\dod~ X(0+)(b-f(0+)),\quad t\to\infty.$$
\end{lemma}

\vskip0.5cm \noindent {\bf Acknowledgement}. The work of A. Iksanov was supported by the National Research Foundation of Ukraine (project 2023.03/0059 ‘Contribution to modern theory of random series’).

\end{document}